\documentclass[10pt]{amsart}
\usepackage{amsmath,graphicx,amssymb,amsthm}
\def\A{\mathcal{A}}
\def\B{\mathcal B}
\def\F{\mathcal F}

\def\K{\mathcal K}
\def\P{\mathcal P}

\def\B{\mathcal B}
\def\C{\mathcal C}
\def\H{\mathcal H}
\def\K{\mathcal K}

\def\M{\mathcal{M}}

\def\S{\mathcal S}

\def\amslatex{$\mathcal{A}\kern-.1667em\lower.5ex\hbox{$\mathcal{M}$}\kern-.125em\mathcal{S}$-\LaTeX}

\newtheorem{set}{set}[section]
\newtheorem{Corollary}[set]{Corollary}
\newtheorem{Definition}[set]{Definition}
\newtheorem{Example}[set]{Example}

\newtheorem{Lemma}[set]{Lemma}

\newtheorem{Remark}[set]{Remark}
\newtheorem{Theorem}[set]{Theorem}
\newcommand{\define}{\mathrel{\hbox{$\equiv$\hskip -.90em \lower .47ex \hbox{$\leftharpoondown$}}}}
\newcommand{\enifed}{\mathrel{\hbox{$\equiv$\hskip -.90em \lower .47ex \hbox{$\rightharpoondown$}}}}

\begin{document}
\title[Multidimensional  Compound Poisson Processes in Free Probability] {Multidimensional Compound Poisson Distributions in Free Probability}
\thanks{This work is supported partially by the NSFC grants 11101220, 11271199,11671214,Fund 96172373 and the Fundamental Research Funds for the Central Universities.}
\author{Guimei An}
\address{School of Mathematical Sciences and LPMC \\
Nankai University, Tianjin 300071, China}
\email[Guimei
An]{angm@nankai.edu.cn}

\author{Mingchu Gao}
\address{School of Mathematics and Information Science,
Baoji University of Arts and Sciences, Baoji, Shaanxi 721013, China; and
Department of Mathematics,
Louisiana College, Pineville, LA 71359, USA}
\email[Mingchu
Gao]{mingchug@yahoo.com}

\maketitle
\begin{abstract}

 Inspired by R. Speicher's multidimensional free central limit theorem and semicircle families, we prove an infinite dimensional compound Poisson limit theorem in free probability, and define infinite dimensional compound free Poisson distributions in a non-commutative probability space. Infinite dimensional free infinitely divisible distributions are defined and characterized in terms of their free cumulants. It is proved that for   a sequence of random variables, the following three statements are equivalent. (1) The distribution of the sequence is multidimensional free infinitely divisible. (2) The sequence is  the limit in distribution of a sequence of triangular trays of families of random variables. (3) The sequence  has a distribution  same as that of $\{a_1^{(i)}: i=1, 2, \cdots\}$ of a multidimensional free Levy process $\{\{a_t^{(i)}:i=1, 2, \cdots\}: t\ge 0\}$. Under certain technical appumption, this is the case if and only if the sequence   is the limit in distribution of a sequence of sequences of random variables having multidimensional compound free Poisson distributions.
\end{abstract}

\section*{Introduction}
The most popular distributions in classical probability are Gaussian distributions. Poisson distributions  form  a class of the
most prominent distributions in classical probability beyond Gaussian
distributions (Lecture 12 in \cite{NS}). A very important method in classical probability theory for generating distributions centers around a generalization of Poisson distributions. Given a random variable $a$, one can construct out of this a compound Poisson distribution $\pi_a$ by the prescription that the moments of $a$ give, up to a common factor, the cumulants of $\pi_a$. One possibility to make the transition from $a$ to $\pi_a$ is by a limit theorem. The importance of these compound Poisson distributions rises from the fact that all infinitely divisible distributions can be approximated by compounded Poisson distributions (see Chapter XVII in \cite{WF} and Section 4.4 in \cite{RS}). These ideas have been imitated in free probability.
Semicircle distributions are the counterpart of normal distributions in free probability theory. Such a distribution can be realized as the limit distribution of a sequence of certain random variables. This result was called the {\sl free central limit theorem} proved by D. Voiculescu \cite{DV} (see also Theorem 8.10 in \cite{NS}).
Very similarly, a (compound) free Poisson distribution can be realized as the limit in distribution of a sequence of simple distributions (12.11,  12.12,  12.15, and 12.16 in \cite{NS}, also \cite{RS}, \cite{RS1} and \cite{DV1}). More generally, infinitely divisible distributions,  Levy processes,  and limit theorems have been studied thoroughly in free probability  (see, for instance, \cite{BnT} and \cite{BP}). Very recently, the infinite divisibility of the distribution of a pair of random variables was studied in \cite{GHM} and \cite{MG} in the setting of bi-free probability, a new research area in free probability theory introduced by Voiculescu in \cite{DV2}.

In contrast to the thorough study of free Poisson distributions and free infinitely divisible distributions, we cannot find too much work on multidimensional Poisson distributions and multidimensional infinitely divisible distributions in free probability. Inspired by the corresponding ideas in classical probability, R. Speicher \cite{RS} gave a brief theory on operator-valued compound free Poisson distributions and multidimensional free infinitely divisible distributions. Given an $m$-tuple $(a_1, a_2, \cdots, a_m)$ of random variables in a non-commutative probability space $(\A,\varphi)$, and a constant number $\lambda\in \mathbb{R}$, Speicher defined a multidimensional compound free Poisson distribution as follows. An $m$-tuple $\{b_1, b_2, \cdots, b_m\}$ of random variables in a non-commutative probability space $(\B,\phi)$ has a multidimensional compound free Poisson distribution if $\kappa_n(b_{i_1}, b_{i_2}, \cdots, b_{i_n})=\lambda\varphi(a_{i_1}a_{i_2}\cdots a_{i_n})$, for $i_1, \cdots, i_n\in \{1, 2, \cdots, m\}, n\in \mathbb{N}$, where $\kappa_n$ is the $n$-th free cumulant on $(\B, \phi)$ (4.4.1 in \cite{RS}). A multidimensional compound free Poisson limit theorem,  the definition of multidimensional  infinitely divisible distributions, and the semigroup of distributions  and the approximation of compound free Poisson distributions for a multidimensional free infinitely divisible distributions were given in Sections 4.4 and 4.5 in \cite{RS}. All distributions  in \cite{RS} are finite dimensional, that is, they are distributions of $m$-tuples of random variables.
 Benaych-Georges \cite{BG} characterized  an $m$-dimensional free infinitely divisible distribution in terms of its free cumulants under the hypothesis that all random variables live in a tracial non-commutative probability space $(\A,\varphi)$ (i. e., $\varphi(ab)=\varphi(ba)$, for all $a, b\in \A$).

 R. Speicher \cite{RS1}  gave a multidimensional central limit theorem (see also Theorem 8.17 in \cite{NS}). Roughly speaking, the theorem states that the (joint) distribution of a semicircle family can be realized as the limit in distribution of a sequence of families of random variables.
 Based on the same philosophy,  in this paper, we prove an infinite dimensional compound free Poisson limit theorem, and  develop a theory on infinite dimensional infinitely divisible distributions in free probability.

 This paper is organized as follows. Beside this introduction, there are three sections in the paper. In Section 1, we present some concepts and results in free probability used in sequel. Section 2 is devoted to the study of infinite dimensional compound free Poisson distributions.  A free Poisson random variable can be realized as the limit in distribution of a sequence of free triangular arrays of scalar multiples of projections (see the arguments at the beginning of Section 2). Substituting the projections in the free Poisson limit theorem by sequences of projections (with possible distinct expectations), we get an infinite dimensional free Poisson limit theorem.  By using techniques in  ultra-products of $C^*$-algebras, we can choose the limit sequence of random variables from a $C^*$-probability space (Theorem 2.5).
 Using other techniques in tensor products of $C^*$-algebras, we then present an infinite dimensional compound free Poisson limit theorem (Theorem 2.8). We say that the limit sequence of random variables has a {\sl multidimensional compound free Poisson distribution} (Definition 2.10). In Example 2.11, we present examples of sequences of random variables having multidimensional compound free Poisson distributions such as Speicher's multidimensional compound free Poisson distributions given in Section 4.4 of \cite{RS}, tuples of random variables constructed from general self-adjoint random variables and a free family of semicircle random variables (a special type of such constructions can be found in Proposition 12.18 in \cite{NS}), and free families of compound free Poisson random variables with certain distributions.
  In Section 3, we study infinite dimensional infinitely divisible distributions in free probability.  We generalize Speicher's work in Section 4.5 of \cite{RS} and Benaych-Georges' work in \cite{BG} in two aspects. We study distributions of sequences of random variables (we thus call such a distribution an infinite dimensional distribution).  Random variables under study are in a $C^*$-probability space (and its linear functional is not necessarily  tracial). We first give the definitions of infinite dimensional infinitely divisible distributions and infinite dimensional free Levy processes $\{\{a_t^{(i)}:i=1, 2, \cdots\}:t\ge 0\}$ (Definitions 3.1 and  3.2).  We characterize the free infinite divisibility of the distribution of a sequence of random variables in terms of its free cumulants (Theorem 3.3 (2) and (4)).
  Finally, under certain technical assumption, using the techniques in inductive limits of $C^*$-algebras, we  generalize Speicher's free Poisson approximation theorem 4.5.5 in \cite{RS} to the infinite dimensional distribution case (Theorem 3.6).

\section{Preliminaries}

In this section we recall some basic concepts and results in free probability used in sequel. The reader is referred to \cite{NS} and \cite{VDN} for more details on free probability,  and to \cite{KR} for operator algebras.

{\bf Non-commutative Probability spaces}. A non-commutative probability space is a pair $(\A,\varphi)$ consisting of a unital algebra $\A$ and a unital linear functional $\varphi$ on $\A$. When
$\A$ is a $*$-unital algebra, $\varphi$ should be positive, i. e.,
$\varphi(a^*a)\ge 0,\forall a\in \A$. A $C^*$-probability space
$(\A,\varphi)$ consists of a unital $C^*$-algebra and a state
$\varphi $ on $\A$. A $W^*$-probability space $(\A,\varphi)$
consists of a finite von Neumann algebra $\A$ and a faithful normal
tracial state $\varphi$ on $\A$. An element $a\in \A$ is called a
{\sl (non-commutative) random variable}. $\varphi (a^n)$ is called the
{\sl $n$-th moment} of $a$, for $n=1, 2, \cdots$. Let $\mathbb{C}[X]$ be the
complex algebra of all polynomials of an indeterminate $X$. The
linear function $\mu_a:\mathbb{C}[X]\rightarrow \mathbb{C}$,
$\mu_a(P(X))=\varphi (P(a)), \forall P\in \mathbb{C}[X]$, is called
the {\sl distribution (or law)} of $a$.

{\bf Joint Distributions.}  Let  $(\A, \varphi)$ be a non-commutative probability space and  $I$ be an index set. For a family $\{a_i \in \A: i\in I\}$, the family $\{\varphi (a_{i_1}a_{i_2}\cdots a_{i_n}): 1\le i_1\le i_2\le \cdots \le i_n\in I, n\ge 1\}$ is called the  family of {\sl joint moments} of $\{a_i:i\in I\} $.
Let $\mathbb{C}\langle X_i:i\in I\rangle$ be the unital algebra freely generated by  non-commutative indeterminates  $X_i, i\in I$. The linear functional $\mu:\mathbb{C}\langle X_i:i\in I \rangle \rightarrow \mathbb{C}$ defined by $$\mu(P)=\varphi (P(a_{i_1}, a_{i_2}, \cdots, a_{i_n})), \forall P=P(X_{i_1}, X_{i_2}, \cdots, X_{i_n})\in \mathbb{C}\langle X_i: i\in I\rangle,$$ is called the {\sl joint distribution} of the sequence $\{a_1:i\in I\}$. A sequence $\{\{a_{i,n}:i\in I\}:n=1, 2, \cdots \}$ of families of random variables in a non-commutative probability space $(\A, \varphi)$ {\sl converges in distribution} to a family $\{b_i:i\in I\}$ of random variables in a non-commutative probability space $(\B,\phi)$ if, for all $i_1, i_2, \cdots, i_m\in I$, $m\in \mathbb{N}$, $$\lim_{n\rightarrow \infty}\varphi(a_{i_1,n}a_{i_2,n}\cdots a_{i_m,n})=\phi(b_{i_1}b_{i_2}\cdots b_{i_m}). $$

{\bf Free independence.}   A  family $\{\A_i: i\in I\}$ of unital subalgebras of a non-commutative probability space $(\A,\varphi)$ is {\sl freely independent (or free)} if $\varphi (a_1a_2\cdots a_n)=0$ whenever the following conditions are met: $a_i\in \A_{l(i)}$, $\varphi (a_i)=0$ for $i=1, 2, \cdots, n$, and $l(i)\ne l(i+1)$, for $i=1, 2, \cdots, n-1$. A family $\{a_i:i\in I\}$ of elements is free if the unital subalgebras generated by $a_i$'s are free.

 {\bf Non-crossing partitions.} Given a natural number $m\ge 1$, let $[m]=\{1, 2, \cdots, m\}$. A {\sl partition} $\pi$ of $[m]$ is a collection of non-empty disjoint subsets of $[m]$ such that the union of all subsets in $\pi$ is $[m]$. A partition $\pi=\{B_1, B_2,\cdots, B_r\}$ of $[m]$ is {\sl non-crossing} if one cannot find two block $B_i$ and $B_j$ of $\pi$, and four numbers $p_1, p_2\in B_i$, $q_1, q_2\in B_j$ such that $p_1<q_1<p_2<q_2$. The collection of all non-crossing partitions of $[m]$ is denoted by $NC(m)$.
 $|NC(m)|$, the number of non-crossing partitions of $[m]$,  is $C_m=\frac{(2m)!}{m!(m+1)!}$, which is called the $m$-th Catalan number (Notation 2.9 in \cite{NS}).

 {\bf The Mobius function.} Let $P$ be a finite partial ordered set (poset),  and $P^{(2)}=\{(\pi, \sigma): \pi, \sigma\in P, \pi \le \sigma\}$. For two functions $F, G:P^{(2)}\rightarrow \mathbb{C}$, we define the convolution $F*G$ by $$F*G(\pi, \sigma):=\sum_{\rho\in P, \pi\le \rho\le \sigma}F(\pi, \rho)G(\rho, \sigma).$$ Let $\delta(\pi,\sigma)=1$, if $\pi=\sigma$; $\delta(\pi, \sigma)=0,$ if $\pi <\sigma$. Then
 $$F*\delta(\pi,\sigma)=\sum_{\rho\in P, \pi\le \rho\le \sigma}F(\pi, \rho)\delta(\rho, \sigma)=F(\pi, \sigma), \forall F.$$ It follows that $\delta$ is the unit of set of all functions on $P^{(2)}$ with respect to convolution $*$.  The inverse function of the function $\zeta: P^{(2)}\rightarrow \mathbb{C}$, $\zeta(\pi, \sigma)=1, \forall (\pi, \sigma)\in P^{(2)}$, with respect to the convolution $*$ is called the Mobius function $\mu_P$ of $P$.

 {\bf Free cumulants.} Let $\pi, \sigma\in NC(n)$. We say $\pi\le \sigma$ if each block (a subset of $[n]$) of $\pi$ is completely contained in one of the blocks of $\sigma$. $NC(n)$ is a poset by this partial order. The Mobius function of $NC(n)$ is denoted by $\mu_n$. The unital linear functional $\varphi:\A\rightarrow \mathbb{C}$ produces a sequence of multilinear functionals $$\varphi_n:\A^n\rightarrow \mathbb{C}, \varphi_n(a_1, a_2, \cdots, a_n)=\varphi (a_1a_2\cdots a_n), n=1, 2, \cdots.$$ Let $V=\{i_1, i_2, \cdots, i_s\}\subseteq [n]$ such that $i_1<i_2<\cdots <i_s$. We define $\varphi_V(a_1, a_2, \cdots, a_n)=\varphi (a_{i_1}a_{i_2}\cdots a_{i_s})$. More generally, for a partition $\pi=\{V_1, V_2, \cdots, V_r\}\in NC(n)$, we define $$\varphi_\pi(a_1, a_2, \cdots, a_n)=\prod_{i=1}^r\varphi_{V_i}(a_1, a_2, \cdots, a_n).$$ The $n$-th {\sl free cumulant} of $(\A,\varphi)$ is the multilinear functional $\kappa_n:\A^n\rightarrow \mathbb{C}$ defined by   $$\kappa_n(a_1, a_2, \cdots, a_n)=\sum_{\pi\in NC(n)}\varphi_\pi (a_1, a_2, \cdots, a_n)\mu_n(\pi, 1_n), $$ where $1_n=[n]$ is the single-block partition of $[n]$. The convergence in distribution of a sequence  of families of random variables  can be characterized by their free cumulants. A sequence  $\{\{a_{i,n}:i\in I\}:n=1, 2, \cdots \}$ converge in distribution to $\{b_i:i\in I\}$ if and only if for all $i_1, i_2, \cdots, i_m\in I$, $m\in \mathbb{N}$, $$\lim_{n\rightarrow \infty}\kappa_m(a_{i_1,n},a_{i_2,n}, \cdots, a_{i_m,n})=\kappa_m(b_{i_1},b_{i_2},\cdots, b_{i_m}). $$

 Free cumulants $\kappa_n:\A^n\rightarrow \mathbb{C}$ and free independence have a very beautiful relation.
 \begin{Theorem}[Theorem 11.20 in \cite{NS}] A family $\{a_i:i\in I\}$ of elements in $(\A, \varphi)$ is freely independent if and only if for all $n\ge 2$ and all $i(1), i(2), \cdots, i(n)\in I$, $$\kappa_n(a_{i(1)}, a_{i(2)}, \cdots,  a_{i(n)})=0$$ whenever there exist $1\le l, k\le n$ with $i(l)\ne i(k)$.
 \end{Theorem}

 {\bf Semicircle families.} Let $(\A,\varphi)$ be a $*$-probability space. A self-adjoint element $a\in \A$ is a {\sl semicircle element} (or has a {\sl semicircle distribution}) if $$\varphi (a^n)=\frac{2}{\pi r^2}\int_{-r}^rt^n\sqrt{r^2-t^2}dt, n=1, 2, \cdots,$$ where $r$ is called the radius of the distribution of  $a$. When $r=2$, $\varphi (a^2)=1$, we call $a$ a {\sl standard semicircle} element (or has a standard semicircle distribution). A semicircle element can be characterized by its moments $\varphi(a^{2k})=(r^2/4)^kC_k$, where $C_k$ is the $k$-th Catalan number, and $\varphi (a^{2k+1})=0$, $k=0, 1, 2, \cdots$, or by its free cumulants $\kappa_n(a):=\kappa_n(a,\cdots, a)=\delta_{n,2}\frac{r^2}{4}$ ((11.13) in \cite{NS}). Given an index set $I$ and a positive definite matrix $(c_{ij})_{i,j\in I}$, a family $(s_i)_{i\in I}$ of self-adjoint random variables in a $*$-probability space $(\A, \varphi)$ is called a {\sl semicircular family of covariance $(c_{i,j})_{i,j\in I}$} if $\kappa_n(s_{i(1)}, \cdots, s_{i(n)})=\delta_{n,2}c_{i(1), i(2)}$, for $i(1), \cdots, i(n)\in I$. (See 8.15 in \cite{NS}.)

\section{Multidimensional free Poisson distributions}

By the discussion in Page 203 and Exercise 12.22 of \cite{NS}, a classical Poisson distribution is the limit in distribution of a sequence of convolutions of Bernoulli distributions. In the point of view of random variables, we can restate it as follows. Let $\lambda >0, \alpha\in \mathbb{R}$. For each $N\in \mathbb{N}, N>\lambda$, let $\{b_{i,N}:i=1,2, \cdots, N\}$ be a sequence of i.i.d. Bernoulli random variables such that $$Pr(b_{i,N}=0)=1-\frac{\lambda}{N}, Pr(b_{i,N}=\alpha)=\frac{\lambda}{N}.$$ Then the binomial random variable $S_N=\sum_{i=1}^Nb_{i,N}$ has a binomial distribution $$Pr(S_N=k\alpha)=\binom{N}{k}(\frac{\lambda}{N})^k(1-\frac{\lambda}{N})^{N-k},$$ $k=0, 1, 2, \cdots, N$,  where $\binom{N}{k}$ is the combination number (or the binomial coefficient). Let $N\rightarrow \infty$, by elementary calculus, we can get
$$\lim_{N\rightarrow \infty}Pr(S_N=k\alpha)=\frac{\lambda^k}{k!}e^{-\lambda}=Pr(P=k\alpha),$$ where $P$ has a Poisson distribution $Pr(P=k\alpha)=\frac{\lambda^k}{k!}e^{-\lambda}$, $k=0, 1, 2, \cdots$.

In the non-commutative case, the free Poisson limit theorem (Proposition 12.11 in \cite{NS}) states that a free Poisson distribution is the limit in distribution of a sequence of free convolutions of Bernoulli distributions. Let's restate it in the language of random variables.

  Let $(\A,\varphi)$ be a non-commutative probability space.  A {\sl Bernoulli random variable} $a\in \A$ is a linear combination $a=\alpha p+\beta (1-p)$, where $\alpha,\beta \in \mathbb{R}$, and $p\in \A$ is an idempotent (i. e., $p^2=p$) with $0\le \varphi(p)\le 1$. The classical interpretation of a Bernoulli random variable is that $a$ is a random variable with two ``values": $\alpha$ and $\beta$, and $Pr(a=\alpha)=\varphi(p), Pr(a=\beta)=1-\varphi(p)$. In the free Poisson limit theorem,  $\beta=0$, $\varphi(p)=\frac{\lambda}{N}, N>\lambda$. We can restate the free Poisson limit theorem as follows. Let $\lambda>0, \alpha\in \mathbb{R}$. For $N\in \mathbb{N}, N>\lambda$, let $\{\alpha p_{1,N}, \alpha p_{2,N}, \cdots, \alpha p_{N,N}\}$ be a free family of Bernoulli random variables such that $\varphi(p_{i,N})=\frac{\lambda}{N}, i=1, 2, \cdots, N$. Let $S_N=\sum_{i=1}^N\alpha p_{i,N}$. Then
 $$\lim_{N\rightarrow \infty}\kappa_m(S_N)=\lambda \alpha^m, m=1,2, \cdots.$$
Hence, we can restate the definition of free Poisson distributions as follows.
\begin{Definition}[Proposition 12.11, Definition 12.12 \cite{NS}] Let $\lambda\ge 0, \alpha \in \mathbb{R}$, and $(\A,\varphi)$ be a non-commutative probability space. A random variable $a\in \A$  has a free Poisson distribution if the free cumulants of $a$ are $\kappa_n(a)=\lambda \alpha^n, \forall n\in \mathbb{N}$.
\end{Definition}

In this section, we shall generalize the results on free Poisson distributions in Lecture 12 of \cite{NS} to the multidimensional case.

By the proof of Theorem 13.1  in \cite{NS}, we can modify the theorem slightly as follows.
\begin{Theorem}[Theorem 13.1 and Lemma 13.2 in \cite{NS}]
Let $\{n_k\}$ be a sequence of natural numbers such that
$\lim_{k\rightarrow \infty}n_k =\infty$, and, for each natural
number $k$, $(\A_k, \varphi_k)$ be a non-commutative probability
space. Let $I$ be an index set. Consider a triangular array of
random variables, i. e., for each $i\in I$, $0\le r\le n_k$, we have
a random variable $a^{(i)}_{r, n_k}\in \A_k$. Assume that, for each
$k$, the sets $\{a^{(i)}_{1, n_k}\}_{i\in
I},\{a^{(i)}_{2, n_k}\}_{i\in I}, \cdots,
\{a^{(i)}_{n_k,n_k}\}_{i\in I} $ are free and identically
distributed. Then the following statements are equivalent.
\begin{enumerate}
\item There is a family of random variables $(b_i)_{i\in I}$ in some non-commutative probability space $(\A,\varphi)$ such that
$(a^{(i)}_{1, n_k}+a^{(i)}_{2, n_k}+\cdots+a^{(i)}_{n_k,n_k})_{i\in I}$ converges in distribution to $(b_i)_{i\in I}$, as $k\rightarrow \infty$.
\item For all $n\ge 1$, and all $i(1), i(2), \cdots, i(n)\in I$, the limits
$\lim_{k\rightarrow \infty}n_k\varphi_k(a^{(i(1))}_{r, n_k}\cdots a^{(i(n))}_{r, n_k})$ exist, $1\le r\le n_k$.
\item For all $n\ge 1$, and all $i(1), i(2), \cdots, i(n)\in I$, the limits
$\lim_{k\rightarrow \infty}n_k \kappa_n^k(a^{(i(1))}_{r, n_k},\cdots, a^{(i(n))}_{r, n_k})$ exist, $1\le r\le n_k,$ where $\kappa_n^k$ is the $n$-th free cumulant of $\A_k$.
\end{enumerate}
Furthermore, if one  of these conditions is satisfied, then the
limits in $(2)$ are equal to the corresponding limits in $(3)$, and
the joint distribution of the limit family $(b_i)_{i\in I}$ is
determined  by, for $n\ge 1, i(1), i(2),
\cdots, i(n)\in I$,
$$\kappa_n (b_{i(1)}b_{i(2)}\cdots b_{i(n)})=\lim_{k\rightarrow \infty}n_k \varphi_k (a^{(i(1))}_{r, n_k}a^{(i(2))}_{r, n_k}\cdots a^{(i(n))}_{r, n_k}). $$
\end{Theorem}

We will use the following elementary results in sequel.
\begin{Lemma}Let $\{a_{i,j}: i, j=1, 2, \cdots\}$ be a bi-index sequence of complex numbers. If $\sup\{|a_{i,j}|:i=1,2,\cdots\}=M_j<\infty, \forall j$, then there exists a sequence $(n_k)_{k\in \mathbb{N}}$ of natural numbers such that $\lim_{k\rightarrow \infty}n_k=\infty$, and $\lim_{k\rightarrow \infty}a_{n_k,j}$ exists,$\forall j\in \mathbb{N}$.
\end{Lemma}

 We need to review some basic concepts on ultrafilters of $\mathbb{N}$ (See \cite{GH} and Appendix A of \cite{SS} for details).  A {\sl filter} $\mathcal{F}$ is a non-empty set of subsets of $\mathbb{N}$ satisfying
 \begin{enumerate}
 \item if $X\in \mathcal{F}$ and $X\subseteq Y\subseteq \mathbb{N}$, then $Y\in \mathcal{F}$;
 \item if $X, Y\in \mathcal{F}$, then $X\cap Y\in \F$;
 \item $\emptyset $ is not in $\F$.
 \end{enumerate}
  A filter $\F$ is an {\sl ultrafilter} if for $\forall X\subseteq \mathbb{N}$, we have $X\in \F$ or $\mathbb{N}\setminus X\in \F$. An ultrafilter is called a free ultrafilter if $\F$ does not contain any finite sets. A filter $\F$ is {\sl infinite} if $V$ is an infinite subset, $\forall V\in \F$. Let $\beta \mathbb{N}$ be the Stone-Cech compactification of the set $\mathbb{N}$. Free ultrafilters can be identified as points in $\beta\mathbb{N}\setminus \mathbb{N}$.
  For a free ultrafilter $\omega \in \beta\mathbb{N}\setminus \mathbb{N}$, a sequence $a=(a_n)_n\in l^\infty(\mathbb{N})$, and a number $l\in \mathbb{C}$, we say that $\lim_{n\rightarrow \omega}a_n=l$ if for $\forall \varepsilon >0$, we have $\{n:|a_n-l|<\varepsilon\}\in \omega$.

 \begin{Lemma} Let $x=(x_n)\in l^\infty (\mathbb{N})$ and $l\in \mathbb{C}$. Then there is a subsequence $\{n_k:k=1, 2, \cdots\}$ of natural numbers such that $\lim_{k\rightarrow \infty} x_{n_k}=l$ if and only if there is a free ultrafilter $\omega\in \beta\mathbb{N}\setminus \mathbb{N}$ such that $\lim_{n\rightarrow \omega}x_n=l$.
 \end{Lemma}
 \begin{proof} If $\lim_{n\rightarrow \omega}x_n=l$, by the definition, for any $\varepsilon>0$, $\{n:|x_n-l|<\varepsilon\}\in \omega$. Let $F_k=\{n:|x_n-l|<\frac{1}{k}\}\in \omega$. There exists $n_1\in F_1$. Since $F_2=\{n: |x_n-l|<\frac{1}{2}\}\in \omega$ is an infinite set, there is an $n_2\in \{n>n_1: n\in F_2\}$. Inductively, we get sequence $n_1<n_2<\cdots <n_k<$ such that $n_k\in F_k$, that is, $|x_{n_k}-l|<\frac{1}{k}$, for $k=1, 2, \cdots$. It follows that $\lim_{k\rightarrow \infty}x_{n_k}=l$.

 Conversely, if $\lim_{k\rightarrow\infty}x_{n_k}=l$, then $\mathcal{F}_0=\{\{n:|x_n-l|<\frac{1}{k}\}: k=1, 2, \cdots\}$ is a set of subsets of $\mathbb{N}$ having the finite intersection property (i.e., the intersection of any finitely many subsets in $\mathcal{F}_0$ is not empty). Moreover, any such an intersection is an infinite set in $\F_0$. Let $$\F=\{X\subseteq \mathbb{N}:\exists k_0\in \mathbb{N}, \{n:|x_n-l|<\frac{1}{k_0}\}\subseteq X\}.$$ It is obvious that $\F$ is an infinite filter. The set $\mathfrak{F}$ of all infinite filters $\S$ of $\mathbb{N}$ containing $\F$ is a poset (partial ordered set) with respect to the inclusion of sets. By Zorn's lemma, there is a maximal filter $\omega$ in $\mathfrak{F}$, which is a free ultrafilter (see \cite{AK} for the details).  For $\forall \varepsilon>0$, there exists a $k_0\in \mathbb{N}$ such that $\frac{1}{k_0}<\varepsilon$. Therefore,  $\{n:|x_n-l|<\varepsilon\}\supseteq \{n:|x_n-l|<\frac{1}{k_0}\}$. It follows that $\{n:|x_n-l|<\varepsilon\}\in \F\subseteq \omega$, that is, $\lim_{n\rightarrow\omega}x_n=l$.
 \end{proof}
 \begin{Theorem} Let $\{\alpha_i:i=1,2,\cdots\}$ be a sequence of real numbers, $\{\lambda_i> 0\}_{i\in\mathbb{N}}$ with $\lambda=\sup\{\lambda_i:i\ge 1\}<\infty$,  and for each $N\in \mathbb{N}, N\ge\lambda$, there be $N$ freely independent and identically distributed sequences $$\{p_{1,N}^{(i)}\}_{i\in \mathbb{N}}, \{p_{2,N}^{(i)}\}_{i\in \mathbb{N}}, \cdots, \{p_{N,N}^{(i)}\}_{i\in \mathbb{N}}$$ of  projections on a $C^*$-probability space $(\A_N,\varphi_N)$. Moreover, $\varphi_N(p^{(i)}_{r,N})=\frac{\lambda_i}{N}, i=1,2,\cdots, r=1,2,\cdots, N$.  Define a triangular family of sequences of random variables $\{a_{j, N}^{(i)}=\alpha_ip^{(i)}_{j,N}:i=1, 2, \cdots\}$, for $ j=1, 2, \cdots, N, N\in \mathbb{N}, N\ge \lambda.$ Then there exists a family of random variables $(b_i)_{i\in\mathbb{N}}$ in a $C^*$-probability space $(\A,\varphi)$ and a sequence $\{n_k:k=1,2,\cdots\}$ of natural numbers such that $\lim_{k\rightarrow \infty}n_k=\infty$ and  $(a_{1, n_k}^{(i)}+a_{2,n_k}^{(i)}+\cdots +a_{n_k, n_k}^{(i)})_{i\in \mathbb{N}}$ converges to $(b_i)_{i\in \mathbb{N}}$ in distribution, as $k\rightarrow \infty$. Moreover, for $i_1, i_2, \cdots, i_n, n\in \mathbb{N}$,
 $$\kappa_n(b_{i_1}, b_{i_2}, \cdots, b_{i_n})=\alpha_{i_1}\alpha_{i_2}\cdots \alpha_{i_n}\lim_{k\rightarrow \infty}n_k\varphi(p_{1, n_k}^{(i_1)}p_{1,n_k}^{(i_2)}\cdots p_{1, n_k}^{(i_n)}).$$
 \end{Theorem}
\begin{proof}
For $N, i(1), i(2), \cdots, i(n), n\in \mathbb{N}$, let $$f(N, i(1), i(2), \cdots, i(n))=N\varphi_N(a^{(i(1))}_{r,N}a^{(i(2))}_{r,N}\cdots a^{(i(n))}_{r,N}), 1\le r\le N, $$
and $M(i(1), i(2),\cdots, i(n))= |\alpha_{i(1)}\alpha_{i(2)}\cdots\alpha_{i(n)}| \lambda_{i(1)}^{1/2}\lambda_{i(n)}^{1/2}.$ Then
\begin{align*}
&|f(N,i(1), i(2), \cdots, i(n))|=N|\alpha_{i(1)}\alpha_{i(2)}\cdots\alpha_{i(n)}\varphi_N(p^{(i(1))}_{r,N}p^{(i(2))}_{r,N}\cdots p^{(i(n))}_{r,N})|\\
\le & N |\alpha_{i(1)}\alpha_{i(2)}\cdots\alpha_{i(n)}|\varphi_N(p^{(i(1))}_{r,N})^{1/2}\varphi_N(p^{(i(n))}_{r,N}p^{(i(n-1))}_{r,N}\cdots p^{(i(3))}_{r,N}p^{(i(2))}_{r,N}p^{(i(3))}_{r,N}\cdots p^{(i(n))}_{r,N})^{1/2}\\
\le & N|\alpha_{i(1)}\alpha_{i(2)}\cdots\alpha_{i(n)}|\frac{\lambda_{i(1)}^{1/2}}{N^{1/2}}\varphi_N(p^{(i(n))}_{r,N}p^{(i(n-1))}_{r,N}\cdots p^{(i(4))}_{r,N}p^{(i(3))}_{r,N}p^{(i(4))}_{r,N}\cdots p^{(i(n))}_{r,N})^{1/2} \\
\le& \cdots\\
\le & N|\alpha_{i(1)}\alpha_{i(2)}\cdots\alpha_{i(n)}|\frac{\lambda_{i(1)}^{1/2}}{N^{1/2}}\varphi_N(p^{(i(n))}_{r,N})^{1/2}\\
=&N|\alpha_{i(1)}\alpha_{i(2)}\cdots\alpha_{i(n)}|\frac{\lambda_{i(1)}^{1/2}}{N^{1/2}}\frac{\lambda_{i(n)}^{1/2}}{N^{1/2}}\\
=&M(i(1), i(2),\cdots, i(n)),
 \end{align*}
where the first inequality holds because of the Cauchy-Schwartz inequality $|\varphi(ab)|\le \varphi(a^*a)^{\frac{1}{2}}\varphi(b^*b)^{\frac{1}{2}}$, the second inequality holds by the inequality $\varphi(a^*ca)\le \varphi(a^*da)$, if $\varphi$ is a positive linear functional,  $c\le d$ are self-adjoint, and $a, b, c, d$ are elements in a unital $C^*$-algebra.

 Let $$S_m=\{(i(1), i(2), \cdots, i(n)): i(1)+i(2)+\cdots +i(n)=m, i(1), i(2), \cdots, i(n)\in \mathbb{N}\},$$ for $m\in \mathbb{N}$. Then for each $ m\in \mathbb{N}$, $S_m$ is a finite set with $|S_m|=k_m$, and $\{S_m:m\in \mathbb{N}\}$ is a partition of the set $\{(i(1), i(2), \cdots, i(n)): i(1), i(2), \cdots, i(n), n\in \mathbb{N}\}$. Define a bijective map
 $$\gamma: S_1\rightarrow \{1\}, \gamma: S_m\rightarrow \{(\sum_{l=1}^
 {m-1}k_{l})+1, (\sum_{l=1}^{m-1}k_{l})+2, \cdots, \sum_{l=1}^
 {m}k_{l}\}, m\ge 2. $$ For instance,   $\gamma((1,1))=2, \gamma(2)=3, \gamma(S_2)=\{2,3\}$.  It implies that
\begin{align*}
&\gamma(\{(i(1), i(2), \cdots, i(n)): i(1), i(2), \cdots, i(n), n\in \mathbb{N}\})\\
=&\gamma(S_1)\cup \gamma(S_2)\cup\cdots\cup \gamma(S_m)\cup\cdots\\
=&\{j:j=1,2, 3, \cdots\}.
\end{align*}
Thus, $\{f(N,i(1), i(2), \cdots, i(n)):N, i(1), i(2), \cdots, i(n),
n\in \mathbb{N}\}=\{f(N,\gamma^{-1}(j)):N, j\in\mathbb{N}\}$
 is a bi-index sequence. By Lemma 2.3, there is a sequence $(n_k)_{k\in\mathbb{N}}$ of natural numbers such that $\lim_{k\rightarrow \infty} n_k=\infty$, and $f(n_k, i(1), i(2), \cdots, i(n))$ converges  as $k\rightarrow\infty$, for every tuple $(i(1), i(2), \cdots, i(n))$. By Theorem 2.2, there is a family of random variables $(b_i)_{i\in \mathbb{N}}$ in a non-commutative probability space $(\A,\varphi)$ such that
$((a^{(i)}_{1,n_k}+a^{(i)}_{2,n_k}+\cdots+a^{(i)}_{n_k,n_k})_{i\in \mathbb{N}}$ converges to $(b_i)_{i\in \mathbb{N}}$ in distribution. Moreover, for $i_1, i_2, \cdots, i_n, n\in \mathbb{N}$,
 $$\kappa_m(b_{i_1}, b_{i_2}, \cdots, b_{i_n})=\alpha_{i_1}\alpha_{i_2}\cdots \alpha_{i_n}\lim_{k\rightarrow \infty}n_k\varphi(p_{1, n_k}^{(i_1)}p_{1,n_k}^{(i_2)}\cdots b_{1, n_k}^{(i_n)}).$$

Now we shall construct a $C^*$-probability space $(\A, \varphi)$ and $b_i\in \A, i=1, 2, \cdots$, such that $(b_i)_i$ has the limit distribution. Let $s_N^{(i)}=\sum_{j=1}^Na_{j,N}^{(i)}\in \A_N$, for $N> \lambda, i=1, 2, \cdots$.
We have proved that there is a subsequence $\{n_k:k=1, 2, \cdots\}$ such that $\lim_{k\rightarrow \infty}\varphi_{n_k}(s_{n_k}^{i(1)}s_{n_k}^{i(n)}\cdots s_{n_k}^{i(n)})$ converges for $i(1), i(2), \cdots, i(n)\in \mathbb{N}$. By Lemma 2.4, there is a $\omega\in \beta\mathbb{N}\setminus \mathbb{N}$ such that
$$\lim_{N\rightarrow\omega}\varphi_{N}(s_{N}^{i(1)}s_{N}^{i(n)}\cdots s_{N}^{i(n)})=\lim_{k\rightarrow\infty}\varphi_{n_k}(s_{n_k}^{i(1)}s_{n_k}^{i(n)}\cdots s_{n_k}^{i(n)}), \forall i(1), i(2), \cdots, i(n)\in \mathbb{N}.\eqno (1.1)$$ Let $\M=\prod_{N\in \mathbb{N}}\A_N$ be the $l^\infty$-product of $\A_N$'s. Let $I_\omega=\{x=(x_N)\in \M: \lim_{N\rightarrow\omega}\|x_N\|=0\}$. Then $\A=\prod^\omega\A_N=\M/I_\omega$ is a unital $C^*$-algebra, which is called the ultra-product $C^*$-algebra of $\A_N$'s (see Section 2 of \cite{GH} for details). Since $(\varphi_N(x_N))_N\in l^\infty(\mathbb{N})$, $\lim_{N\rightarrow \omega}\varphi_N(x_N)$ exists. We define $\varphi_\omega:\A\rightarrow \mathbb{C}$ by  $$\varphi_\omega(x)=\lim_{N\rightarrow \omega}\varphi_N(x_N), \forall x=(x_N)_N\in \A.$$ We show that $\varphi_\omega$ is a state on $\A$. It is obvious that $\varphi_\omega(1)=\lim_{N\rightarrow\omega}\varphi_N(1_N)=1$ and $$\varphi_\omega(x^*x)=\lim_{N\rightarrow \omega}\varphi_N(x^*_Nx_N)\ge 0, \forall x=(x_N)_N\in \A.$$ Therefore, $(\A, \varphi_\omega)$ is a $C^*$-probability space. Let $b_i=(s_N^{(i)})_N\in \A$, for $i\in \mathbb{N}$. Then $$\varphi_\omega (b_{i(1)}b_{i(2)}\cdots b_{i(n)})=\lim_{N\rightarrow\omega}\varphi_N(s_N^{i(1)}s_N^{i(2)}\cdots s_N^{i(n)}), \forall i(1), i(2), \cdots, i(n)\in \mathbb{N}. \eqno (1.2)$$ By (1.1) and (1.2), the sequence $\{b_i:i\ge 1\}$ has the limit distribution.
\end{proof}

We call the limit distribution in Theorem 2.5 a  multidimensional free Poisson distributions. Precisely, we have the following definition.
\begin{Definition} Let $\{\alpha_i:i=1,2,\cdots\}$ be a sequence of real numbers, $\{\lambda_i>0:i=1, 2, \cdots\}$ with $\lambda=\sup\{\lambda_i:i\ge 1\}<\infty$.  A sequence $\{b_i:i=1, 2, \cdots\}$ of random variables in a non-commutative probability space $(\B, \phi)$ has a multidimensional free Poisson distribution, if there is an $\omega \in \beta\mathbb{N}\setminus \mathbb{N}$, and, for each $N\in \mathbb{N}, N\ge\lambda$, there is a family $\{p_N^{(i)}:  i\in \mathbb{N}\}$ of projections in a $C^*$-probability space $(\A_N, \varphi_N)$ such that $\varphi_N(p_N^{(i)})=\frac{\lambda_i}{N}$,  and $$\kappa_n(b_{i(1)}, b_{i(2)}, \cdots, b_{i(n)})=\alpha_{i(1)}\alpha_{i(2)}\cdots \alpha_{i(n)}\lim_{N\rightarrow \omega}N\varphi_N(p_N^{(i(1))}p_N^{(i(2))}\cdots p_N^{(i(n))}),$$
 for all  $(i(1), i(2), \cdots, i(n))\in \mathbb{N}^n.$
\end{Definition}
\begin{Remark}
\begin{enumerate}
\item By Theorems 2.2 and 2.5, For each $i\in \mathbb{N}$,
$$\kappa_n(b_i)=\lim_{k\rightarrow\infty}n_k\varphi_{n_k}((a^{(i)}_{r,n_k})^n)=\alpha_i^n\lambda_i, n=1, 2, \cdots.$$
Hence, $b_i$ has a free Poisson distribution, for each $i\in \mathbb{N}$.
\item If $\sum_{i=1}^\infty \lambda_i<\infty$, $\{p^{(i)}_{r,N}\}_{i\in \mathbb{N}}$ is an orthogonal sequence of  projections in a $W^*$-probability space, then, for $ N>\sum_{i=1}^\infty \lambda_i, r=1,2,\cdots,N$,
$$\kappa_n(b_{i(1)}b_{i(2)}\cdots b_{i(n)})=\lim_{k\rightarrow \infty}n_k\varphi_{n_k}(\alpha_{i(1)}\cdots\alpha_{i(n)}p^{(i(1))}_{r,n_k}\cdots p^{(i(n))}_{ r,n_k})=0,$$ whenever there are $i(j)\ne i(l), 0\le j,l\le n$. This means that $\{b_i:i\in \mathbb{N}\}$ is a free family of free Poisson random variables. A similar procedure of constructing a free family from an orthogonal one can be found in Example 12.19 in \cite{NS}.
\item Combining above (1) and (2), we get a conclusion. If $\{b_i:i=1, 2, \cdots\}$ is a free sequence of free Poisson random variables with $\kappa_n(a_i)=\alpha_i^n\lambda_i$, for $n, i=1, 2, \cdots$, and $\sum_{i=1}^\infty\lambda_i=\lambda<\infty$, then $\{b_i:i=1, 2, \cdots\}$ has a multidimensional free Poisson distribution.
\end{enumerate}
\end{Remark}

Now we give a multidimensional limit theorem of compound free Poisson distributions.
\begin{Theorem}
Let $\{a_i:i=1, 2, \cdots\}$ be a sequence of self-adjoint operators in a $C^*$-probability space $(\A, \varphi)$ and $\{\lambda_i>0:i=1, 2, \cdots\}$ with $\sup\{\lambda_i:i=1, 2, \cdots\}=\lambda<\infty$. For each $N\in \mathbb{N}$ and $N\ge \lambda$, Let $(\C_N, \psi_N)$ be a $C^*$-probability space having  projections $p_N^{(i)}\in \C_N$ with $\psi_N(p_N^{(i)})=\frac{\lambda_i}{N}$, for $i\in \mathbb{N}$, and $\A_N=\A\otimes\C_N$ be the spacial tensor product of $\A$ and $\C_N$ with a state $\varphi_N= \varphi\otimes \psi_N$. Let $a_N^{(i)}=a_i\otimes p_N^{(i)}$, and $\{\{a_{N, j}^{(i)}=a_{i}\otimes p_{N,j}^{(i)}:i=1, 2, \cdots, \}: j=1, 2, \cdots, N\}$ be a freely independent family of $N$ identically distributed sequences of self-adjoint operators in $\A_N$
such that  the distribution of $\{a_{N, j}^{(i)}:i=1, 2, \cdots\}$ is same as that of $\{a_{N}^{(i)}:i=1, 2, \cdots\}$, for  $j=1, 2, \cdots, N$.
Let, moreover, $s_N^{(i)}=\sum_{j=1}^Na_{N,j}^{(i)}$. Then there is a sequence $\{b_i:i=1, 2, \cdots\}$ of compound free Poisson random variables in a $C^*$-probability space $(\B,\phi)$ and an $\omega\in \beta\mathbb{N}\setminus \mathbb{N}$ such that $\{s_N^{(i)}\}$ converges in distribution to $\{b_i; i=1, 2,\cdots\}$, as $N\rightarrow \omega$, and $$\kappa_n(b_{i(1)}, b_{i(2)}, \cdots, b_{i(n)})=\varphi(a_{i(1)}a_{i(2)}\cdots a_{i(n)})\lim_{N\rightarrow \omega}N\phi_N(p_N^{(i(1))}p_N^{(i(2))}\cdots p_N^{(i(n))}), $$for all  $(i(1), i(2), \cdots, i(n))\in \mathbb{N}^n$.
\end{Theorem}
\begin{Remark}
 The Proof of the above theorem is the same as that of Theorem 2.5. The only change is to let $M(i(1), i(2),\cdots, i(n))= |\varphi(a_{i(1)}a_{i(2)}\cdots a_{i(n)})|\lambda_{i(1)}^{1/2}\lambda_{i(n)}^{1/2}$ in the proof of Theorem 2.5, where $(i(1), i(2), \cdots, i(n))\in \mathbb{N}^n$, $n\in \mathbb{N}$.
\end{Remark}

\begin{Definition} Let $\{a_i:i=1, 2, \cdots\}$ be a sequence of self-adjoint operators in a $C^*$-probability space $(\A, \varphi)$ and $\{\lambda_i>0:i=1, 2, \cdots\}$ with $\lambda=\sup\{\lambda_i:i\ge 1\}<\infty$.  A sequence $\{b_i:i=1, 2, \cdots\}$ of random variables in a non-commutative probability space $(\B, \phi)$ has a multidimensional compound free Poisson distribution, if there is an $\omega \in \beta\mathbb{N}\setminus \mathbb{N}$, and for each $N\in \mathbb{N}, N\ge\lambda$, there is  a family $\{p_N^{(i)}:  i\in \mathbb{N}\}$ of projections in a $C^*$-probability space $(\A_N, \varphi_N)$ such that $\varphi_N(p_N^{(i)})=\frac{\lambda_i}{N}$,  and $$\kappa_n(b_{i(1)}, b_{i(2)}, \cdots, b_{i(n)})=\varphi(a_{i(1)}a_{i(2)}\cdots a_{i(n)})\lim_{N\rightarrow \omega}N\varphi_N(p_N^{(i(1))}p_N^{(i(2))}\cdots p_N^{(i(n))}),$$
 for all  $(i(1), i(2), \cdots, i(n))\in \mathbb{N}^n.$
\end{Definition}
\begin{Example}
\begin{enumerate}
\item For a positive number $\lambda$, and a sequence $\{a_1, a_2, \cdots \}$ of self-adjoint operators in a $C^*$-probability space $(\A, \varphi)$. Let $p_N^{(i)}=p_N$,  for $i=1, 2, \cdots, $, with $\psi(p_N)=\frac{\lambda}{N}$ in Theorem 2.8. Then the limit random variable family $\{b_1, b_2, \cdots, b_m\}$ has the $n$-th  free cumulant $$\kappa_n(b_{i(1)}, b_{i(2)}, \cdots, b_{i(n)})=\lambda\varphi(a_{i(1)}a_{i(2)}\cdots a_{i(n)}),$$ for $i(1), i(2), \cdots, i(n)\in \{1, 2, \cdots\}$, and $n\in \mathbb{N}$. It follows that the sequence $\{b_i:i=1, 2, \cdots\}$ with free cumulants $\kappa_n(b_{i(1)}, \cdots, b_{i(n)})=\lambda \varphi(a_{i(1)}\cdots a_{i(n)})$, for $i(1), i(2), \cdots, i(n)\in \{1, 2, \cdots\}$, and $n\in \mathbb{N}$, has a multidimensional compound free Poisson distribution. When $\{a_i:i=1, 2, \cdots \}$ is a finite sequence (i.e., $a_n=0$, for all $n>m$, for some $m\in \mathbb{N}$), this is the limit theorem 4.4.3 in \cite{RS} when $B=\mathbb{C}$,  and the limit distribution is the (multidimensional) compound $B$-Poisson distribution defined in 4.4.1 in \cite{RS} when $B=\mathbb{C}$.
\item  When $\lambda=1$ in the above example, we get $$\kappa_n(b_{i(1)}, b_{i(2)}, \cdots, b_{i(n)})=\varphi(a_{i(1)}a_{i(2)}\cdots a_{i(n)}),$$ for $i(1), i(2), \cdots, i(n)\in \{1, 2, \cdots, m\}$, and $n\in \mathbb{N}$, which is the $n$-th free cumulant of $\{sa_{i(1)}s, sa_{i(2)}s, \cdots, sa_{i(n)}s\}$, where $s$ is the semicircle random variable with $\varphi(s^2)=1$, and is free from $\{a_n:n\in \mathbb{N}\}$ (see Example 12.19 in \cite{NS}). Therefore, $\{sa_ns:n=1, 2, \cdots, m\}$ constructed in Example 12.19 in \cite{NS} has a multidimensional compound free Poisson distribution.
\item A natural question is to consider the distribution of $\{s_ia_is_i:i=1, 2, \cdots\}$, where $\{s_i:i=1, 2, \cdots\}$ is a semicircle family, and $\{s_i:i=1, 2, \cdots\}$ is free from $\{a_i:i=1, 2, \cdots\}$. Let $b_i=s_ia_is_i$, for $i\in \mathbb{N}$. For $i(1), i(2), \cdots, i(n)\in \mathbb{N}$, by the proof of Proposition 12.18 in \cite{NS}, for $n>1$, we have
\begin{align*}
&\kappa_n(b_{i(1)}, b_{i(2)}, \cdots, b_{i(n)})\\
=&\sum_{\pi\in NC(3n), \pi\vee\sigma=1_{3n}}\kappa_\pi(s_{i(1)}, a_{i(1)}, s_{i(1)}, s_{i(2)}, a_{i(2)}, s_{i(2)},\cdots,s_{i(n)}, a_{i(n)}, s_{i(n)})\\
=&\sum_{\pi\in NC(n)}(\prod_{j=1}^{n-1}\varphi(s_{i(j)}s_{i(j+1)}))\varphi(s_{i(n)}s_{i(1)})\kappa_\pi(a_{i(1)}, a_{i(2)}, \cdots, a_{i(n)})\\
=&(\prod_{j=1}^{n-1}\varphi(s_{i(j)}s_{i(j+1)}))\varphi(s_{i(n)}s_{i(1)})\varphi(a_{i(1)}a_{i(2)} \cdots a_{i(n)}),
\end{align*}
where $\sigma=\{\{1, 2, 3\}, \{4, 5, 6\}, \cdots, \{3n-2, 3n-1, 3n\}\}$.   If $\{s_i:i=1, 2, \cdots\}$ is a free family of semicircle random variables,  then
$\kappa_n(b_{i(1)}b_{i(2)}\cdots b_{i(n)})=0$, if there are two $j_1$ and $j_2$ such that $i(j_1)\ne i(j_2)$. This shows that $\{b_i:i\in \mathbb{N}\}$ is a free family of compound free Poisson random variables.
\item Now consider a finite family $\{s_ia_is_i: i=1, 2, \cdots, m\}$ of random variables, where $\{s_i:i=1, 2, \cdots, m\}$ is a free family of standard semicircle random variables, and the family is free from $\{a_i: 1, 2, \cdots, m\}$. By above (3), $\{b_i=s_ia_is_i: i=1, 2, \cdots, m\}$ is a free family of free compound free Poisson random variables.  We want to show that $\{b_i, i=1, 2, \cdots, m\}$ has a multidimensional compound free Poisson distribution. In fact, for each natural number $N>m$, choose an orthogonal family $\{p_N^{(i)}:i=1, 2, \cdots, m\}$ of projections in a $C^*$-probability space $(\A_N, \varphi)$ such that $\varphi_N(p_N^{(i)})=1/N$. It is obvious that $$\kappa_n(b_{i(1)}, b_{i(2)}, \cdots, b_{i(n)})=\varphi(a_{i(1)}a_{i(2)}\cdots a_{i(n)})\lim_{N\rightarrow \infty}N\varphi_N(p_N^{(i(1))}p_N^{(i(2))}\cdots p_N^{(i(n))}),$$
 for all  $(i(1), i(2), \cdots, i(n))\in \mathbb{N}^n.$ Therefore, by Definition 2.10, $\{b_1, b_2, \cdots, b_m\}$ has a multidimensional compound free Poisson distribution.
\item Let $\{b_i:i=1, 2, \cdots\}$ be a free family of compound free Poisson random variables with $\kappa_n(b_i)=\varphi(a_i^n)\lambda_i$, where $\{a_i:i=1, 2, \cdots\}$ be a sequence of self-adjoint operators in a $C^*$-probability space $(\A, \varphi)$, $\lambda_i>0, i=1, 2, \cdots $, and $\sum_{i=1}^\infty\lambda_i=\lambda<\infty$. We show that $\{b_i:i=1, 2, \cdots\}$ has a multidimensional compound free Poisson distribution. For each $N>\lambda$, let $\{p_N^{(i)}: i=1, 2, \cdots\}$ be an orthogonal sequence of projections in a $W^*$-probability space $(\A, \varphi)$ such that $\varphi_N(p_N^{(i)})=\lambda_i/N$, for $i=1, 2, \cdots$. Then It is obvious that $$\kappa_n(b_{i(1)}, b_{i(2)}, \cdots, b_{i(n)})=\varphi(a_{i(1)}a_{i(2)}\cdots a_{i(n)})\lim_{N\rightarrow \infty}N\varphi_N(p_N^{(i(1))}p_N^{(i(2))}\cdots p_N^{(i(n))}),$$
 for all  $(i(1), i(2), \cdots, i(n))\in \mathbb{N}^n.$ Therefore, by Definition 2.10, $\{b_1, b_2, \cdots\}$ has a multidimensional compound free Poisson distribution.
\end{enumerate}
\end{Example}

\section{Multidimensional Free Infinitely Divisible Distributions}
In this section, we generalize  the results on  multidimensional free infinitely divisible distributions in \cite{RS} and \cite{BG} to a more general case.
\begin{Definition}
 A sequence $\{a_i:i\in \mathbb{N}\}$ of random variables in a  non-commutative probability space $(\A,\varphi)$ has a free infinitely divisible (joint) distribution if for every $k\in \mathbb{N}$, there are $k$ freely independent identically distributed sequences $\{a_{j,k}^{(i)}:i\in \mathbb{N}\}, j=1, 2, \cdots, k$, such that the distribution of $\{\sum_{j=1}^ka_{j,k}^{(i)}:i\in \mathbb{N}\}$ is the same as that of $\{a_i:i\in \mathbb{N}\}$. We call the distribution of $\{a_n:n=1,2,\cdots\}$ an infinite dimensional  free infinitely divisible distribution.
\end{Definition}

A  concept related to multidimensional free infinitely divisible distributions is multidimensional free Levy processes. A finite dimensional version  of the following definition  $\{a_t^{(i)}: t\ge 0, i=1, 2, \cdots, m\}$ is constructed in Section 4.7 in \cite{RS}.
 \begin{Definition} A family $\{\{a_t^{(i)}:i=1, 2, \cdots \}:t\ge 0\}$  of random variables in a non-commutative probability space $(\A,\varphi)$ is a multidimensional free Levy processes if it satisfies the following conditions.
 \begin{enumerate}
 \item For $t>s\ge 0, v>u\ge 0$, $\{a_t^{(i)}-a_s^{(i)}:i=1, 2, \cdots\}$ and $\{a_v^{(i)}-a_u^{(i)}:i=1, 2, \cdots\}$ are free if $(s,t)\cap (u,v)=\emptyset$.
 \item For $t>s\ge 0 $, $\{a_t^{(i)}-a_s^{(i)}:i=1, 2, \cdots\}$ and  $\{a_{t-s}^{(i)}:i=1, 2, \cdots\}$ have the same distribution.
 \item $a_0^{(i)}=0$, for $i=1, 2, \cdots$.
 \item $\{a_t^{(i)}:i=1, 2, \cdots\}$ converges to $0$ in distribution as $t\rightarrow 0$.
 \end{enumerate}
 \end{Definition}

 The following theorem is the main result of this section.

\begin{Theorem} Let $\{a_i:i\in \mathbb{N}\}$ be a sequence of self-adjoint operators in a $C^*$-probability space $(\A,\varphi)$.Then  the first four of the following statements are equivalent. If, furthermore, the linear functional $\varphi$ on $\A$ is tracial (i. e., $\varphi(ab)=\varphi(ba)$, for all $a, b\in \A$), then the following five statements are equivalent.
\begin{enumerate}
\item $\{a_n:n\in \mathbb{N}\}$ has the same distribution as that of $\{a_1^{(i)}:i=1,2, \cdots\}$ of a multidimensional free Levy process $\{\{a_t^{(i)}:i=1, 2, \cdots\}: t\ge 0\}$ of self-adjoint operators in a $C^*$-probability space $(\A, \varphi)$.
\item Let $\P=\mathbb{C}\langle X_i:i\in\mathbb{N}\rangle_0$ be the non-unital $*$-algebra generated by non-commutative indeterminates $X_1, X_2, \cdots$ such that $X_i^*=X_i$ for $i\in \mathbb{N}$.

For $\underline{i}=(i(1), i(2), \cdots, i(m))\in \mathbb{N}^m, \underline{j}=(j(1), j(2), \cdots, j(n))\in \mathbb{N}^n$, let $$\underline{j}_r=(j(n), j(n-1), \cdots, j(1)), X_{\underline{i}}=X_{i(1)}X_{i(2)}\cdots X_{i(m)}, a_{\underline{i}}=(a_{i(1)}, \cdots, a_{i(m)}).$$  Define $\langle X_{\underline{i}}, X_{\underline{j}}\rangle=\kappa_{m+n}(a_{\underline{i}}, a_{\underline{j}_r})$. Then $\langle \cdot, \cdot\rangle$ is a non-negative  sesquilinear form on  $\P$.
\item For every $N\in \mathbb{N}$, there are $N$ freely independent identically distributed sequences $\{a_{j,N}^{(i)}:i\in \mathbb{N}\}, j=1, 2, \cdots, N$, in a $C^*$-probability space $(\A_N, \varphi_N)$,  and a subsequence $\{N_k:k\in \mathbb{N}\}$ of distinct natural numbers such that the distribution of $\{\sum_{j=1}^{N_k}a_{j,N_k}^{(i)}:i\in \mathbb{N}\}$ converges to  the distribution of $\{a_i:i\in \mathbb{N}\}$, as $k\rightarrow \infty$.
\item $\{a_n:n=1, 2, \cdots\}$ has a free infinitely divisible distribution.
\item Define a linear functional $\kappa: \P\rightarrow \mathbb{C}$, $\kappa(P)=\sum\alpha_{i_1, \cdots, i_n}\kappa_n(a_{i_1}, \cdots, a_{i_n})$, for a polynomial $P=\sum\alpha_{i_1, \cdots, i_n}X_{i_1}\cdots X_{i_n}\in \P$. Then $\kappa$ is  positive and tracial, and $\kappa(P^*)=\overline{\kappa(P)}$, for all $P\in \P$.
\end{enumerate}

\end{Theorem}
\begin{proof} The implications from $(4)$ to $(3)$, and from $(1)$ to $(4)$ are obvious.

$(3)\Rightarrow (2)$. By Theorem 2.2, $$\kappa_{m+n}(a_{\underline{i}}, a_{\underline{j}_r})=\lim_{k\rightarrow\infty}N_k\varphi_{N_k}(a_{i(1)}\cdots a_{i(m)}a_{j(n)}a_{i(n-1)}\cdots a_{j(1)}). $$
For a polynomial $P=\sum\alpha_{\underline{i}}X_{\underline{i}}\in\P$, we have
\begin{align*}
\langle P,P\rangle=&\sum\alpha_{\underline{i}}\overline{\alpha_{\underline{j}}}\langle X_{\underline{i}}, X_{\underline{j}}\rangle
=\sum\alpha_{\underline{i}}\overline{\alpha_{\underline{j}}}\kappa_{m+n}(a_{\underline{i}}, a_{\underline{j}_r})\\
=&\lim_{k\rightarrow\infty}\sum\alpha_{\underline{i}}\overline{\alpha_{\underline{j}}}N_k\varphi_{N_k}(a_{i(1)}\cdots a_{i(m)}a_{j(n)}\cdots a_{j(1)})\\
=&\lim_{k\rightarrow\infty}N_k\varphi_{N_k}(P(a)P(a)^*)\ge 0,
\end{align*}
where $P(a)=\sum\alpha_{\underline{i}}a_{i(1)}\cdots a_{i(n)}$. Moreover, for $X_{\underline{i}}=X_{i(1)}X_{i(2)}\cdots X_{i(m)},X_{\underline{j}}=X_{j(1)}X_{j(2)}\cdots X_{j(n)}\in \P$, we have
\begin{align*}\langle X_{\underline{i}}, X_{\underline{j}}\rangle=&\kappa_{m+n}(a_{\underline{i}}a_{\underline{j}_r})
=\lim_{k\rightarrow \infty}N_k\varphi_{N_k}(a_{i(1)}\cdots a_{i(m)}a_{j(n)}\cdots a_{j(1)})\\
=&\overline{\lim_{k\rightarrow \infty}N_k\varphi_{N_k}(a_{j(1)}\cdots a_{j(m)}a_{i(m)}\cdots a_{i(1)})}\\
=&\overline{\kappa_{m+n}(a_{\underline{j}}, a_{\underline{i}_r})}=\overline{\langle X_{\underline{j}}, X_{\underline{i}}\rangle}.
\end{align*}
It follows that $\langle\cdot, \cdot\rangle$ is a non-negative sesquilinear form on $\P$.

$(2)\Rightarrow (1)$. We can get a Hilbert $\H$ space from $(\P, \langle\cdot,\cdot\rangle)$ after dividing out the kernel $\K=\{P\in \P: \langle P, P\rangle=0\}$, and completion. From now on, we will identify elements in $\P$ with their images in $\H$. We adapt the construction in Section 4.7 in \cite{RS} (see also \cite{GSS}). Let $\widehat{\H}=L^2(\mathbb{R}_+)\otimes \H$, and $\F(\widehat{\H})=\mathbb{C}\Omega\oplus \widehat{\H}\otimes\widehat{\H}\oplus\cdots\oplus \widehat{\H}^{\otimes n}\oplus\cdots$, where $\Omega$ is called the vacuum vector of the full Fock space $\F(\widehat{\H})$. Define $\tau: B(\F(\widehat{\H}))\rightarrow \mathbb{C}$, $\tau(T)=\langle T\Omega, \Omega\rangle$, for $T\in B(\F(\widehat{\H}))$. Then $(B(\F(\widehat{\H})), \tau)$ is a $C^*$-probability space.

For $x \in \widehat{\H}$, and $T\in B(\widehat{\H})$, define $l^*(x), l(x)$ and $p(T)$ on $\F(\widehat{\H})$ as follows.
$$l^*(x)\Omega=x, l^*(x)\xi_1\otimes\xi_2\otimes \cdots\otimes\xi_n=x\otimes \xi_1\otimes\cdots\otimes \xi_n, l(x)\Omega=0,$$
 $$l(x)\xi_1\otimes\cdots\otimes \xi_n=\langle \xi_1,x\rangle\xi_2\otimes\cdots\otimes \xi_n, p(T)\Omega=0, p(T)\xi_1\otimes\cdots\otimes \xi_n=T\xi_1\otimes \xi_2\otimes\cdots\otimes \xi_n,$$ for $\xi_1, \xi_2, \cdots \xi_n\in \widehat{\H}$.

 Let $a_0^{(i)}=0$, and $$a_t^{(i)}=t\varphi(a_i)+l(\chi_{(0,t)}\otimes X_i)+l^*(\chi_{(0,t)}\otimes X_i)+p(\lambda(\chi_{(0,t)})\otimes\lambda(X_i))\in B(\F(\widehat{\H})),$$ for $t>0$, where $\lambda(\chi_{(0,t)})$ and $\lambda(X_i)$ be the left multiplication operators of $\chi_{(0,t)}$ and $X_i$ on $L^2(\mathbb{R}_+)$ and $\H$, respectively. Let $$a_{s,t}^{(i)}=a_t^{(i)}-a_s^{(i)}=(t-s)\varphi(a_i)+l(\chi_{(s,t)}\otimes X_i)+l^*(\chi_{(s,t)}\otimes X_i)+p(\lambda(\chi_{(s,t)})\otimes\lambda(X_i)),$$ for $0\le s<t$. For $0\le s<t$ and $0\le u<v$, we have $L^2((s,t))\otimes \H$ is orthogonal to $L^2((u,v))\otimes \H$, as subspaces of $\widehat{\H}$, if $(s,t)\cap (u,v)=\emptyset$. It follows from Theorem 4.6.15 in \cite{RS} that $\{a_{s,t}^{(i)}: i=1, 2, \cdots\}$ and $\{a_{u,v}^{(i)}: i=1, 2, \cdots\}$ are free.

 Let $b^{(i)}_1=(t-s)\varphi(a_i)$, and $b^{(i)}_2=l(\chi_{(s,t)}\otimes X_i)+l^*(\chi_{(s,t)}\otimes X_i)+p(\lambda(\chi_{(s,t)})\otimes\lambda(X_i))$. Then
 \begin{align*}
 \tau(a^{(i(1))}_{s,t}a^{(i(2))}_{s,t}\cdots a^{(i(n))}_{s,t})=&\tau(\sum_{j_1, j_2, \cdots, j_n=1}^2b^{(i(1))}_{j_1}\cdots b^{(i(n))}_{j_n})\\
 =&\sum_{V_1\cup V_2=\{1, 2, \cdots, n\}, V_1\cap V_2=\emptyset}\prod_{k\in V_1}b_1^{(i(k))}\tau(b_2^{(i(k_1))}b_2^{(i(k_2))}\cdots b_2^{(i(k_l))})\\
 =&\sum_{V_1\cup V_2=\{1, 2, \cdots, n\}, V_1\cap V_2=\emptyset}(t-s)^{|V_1|}\prod_{k\in V_1}\varphi(a_{i(k)})\tau(b_2^{(i(k_1))}b_2^{(i(k_2))}\cdots b_2^{(i(k_l))}),
 \end{align*}
 where $k_1<k_2<\cdots<k_l, V_2=\{k_1, k_2, \cdots, k_l\}$. Now we show that $\tau(b_2^{(i(1))}\cdots b_2^{(i(n))})$ is a multiple of a power of $t-s$.
Note that $b_2^{(i)}=l(\chi_{(s,t)})\otimes l(X_{i})+l^*(\chi_{(s,t)})\otimes l^*(X_{i})+p(\lambda(\chi_{(s,t)}))\otimes p(\lambda(X_i))$. Let's use the following notations.
$$b_{2,1}=l(\chi_{(s,t)}), b_{2,2}^{(i)}=l(X_i), b_{3,1}=l^*(\chi_{(s,t)}), b_{3,2}^{(i)}=l^*(X_i),b_{4,1}=p(\lambda(\chi_{(s,t)})), b_{4,2}^{(i)}=p(\lambda(X_i)).$$ Then
$$
\tau(b_2^{(i(1))}\cdots b_2^{(i(n))})=\sum_{j_1, j_2, \cdots, j_n=2}^4\tau(\prod_{l=1}^nb_{j_l, 1}\otimes b_{j_l, 2}^{(i(l))})=\sum_{j_1, j_2, \cdots, j_n=2}^4\tau_1(\prod_{l=1}^nb_{j_l, 1})\tau_2(\prod_{l=1}^n b_{j_l, 2}^{(i(l))}),$$ where $\tau_1$ and $\tau_2$ are vacuum states on the full Fock spaces $\F(L^2(\mathbb{R}_+))$ and $\F(\H)$, respectively.
Note that $\tau_1(\prod_{l=1}^nb_{j_l, 1})\ne 0$ implies that $j_n=3, j_1=2$. Therefore, $$\tau_1(\prod_{l=1}^nb_{j_l, 1})=\langle b_{j_2, 1}\cdots b_{j_{n-1}, 1}\chi_{(s,t)}, \chi_{(s,t)}\rangle.$$ Also, $b_{4,1}\chi_{(s,t)}=\chi_{(s,t)}^2=\chi_{(s,t)}.$ Therefore, we can assume that there is no $b_{4,1}$ in $\{b_{j_2, 1}, \cdots, b_{j_{n-1}, 1}\}$. Since $\tau_1(\prod_{l=1}^nb_{j_l, 1})=\langle b_{j_2, 1}\cdots b_{j_{n-1}, 1}\chi_{(s,t)}, \chi_{(s,t)}\rangle\ne 0$, the numbers of $b_{2,1}$'s must be equal to the number of $b_{3,1}$'s in $\{b_{j_2, 1}, \cdots, b_{j_n,1}\}$. Note also that $b_{2,1}b_{3,1}\xi=(t-s)\xi$, and $b_{3,1}b_{2,1}\chi_{(s,t)}\otimes \xi=(t-s)\chi_{(s,t)}\otimes \xi$, for all $\xi\in \F(L^2(\mathbb{R}_+))$. Therefore, $\tau_1(\prod_{l=1}^nb_{j_l, 1})$ is a multiple of a power of  $t-s$, if $\tau_1(\prod_{l=1}^nb_{j_l, 1})\ne 0$. It follows that $\tau(b_2^{(i(1))}\cdots b_2^{(i(n))})$ is a multiple of a power of $t-s$. Hence, the distribution of $\{a_{s,t}^{(i)}:i\in \mathbb{N}\}$ is the same as that of $\{a_{0,t-s}^{(i)}:i\in \mathbb{N}\}$. The above argument also shows that $\lim_{t-s\rightarrow 0}\tau(a^{(i(1))}_{s,t}a^{(i(2))}_{s,t}\cdots a^{(i(n))}_{s,t})=0$, that is, $\{a_t^{(i)}:i=1, 2, \cdots\}$ converges to $0$ in distribution, as $t\rightarrow 0$. Hence, $\{\{a_t^{(i)}:i=1, 2, \cdots\}: t\ge 0\}$ is a multidimensional free Levy process.

Now we show that $a_1:=\{a_1^{(i)}:i=1, 2, \cdots\}$ has the same distribution as that of $\{a_n:n=1, 2, \cdots\}$. Since $\{a_t:=\{a_t^{(i)}:i\in \mathbb{N}\}: t\ge0\}$ is a multidimensional free Levy process, $\{a_1^{(i)}: i\in \mathbb{N}\}$ has a multidimensional free infinitely divisible distribution $\mu_{a_1}$.  Let $\mu_t=\mu(a_t)$, the distributions of $a_t$, for $t\ge 0$. Then $\mu_0=\mu(a_0)=0$, $\mu_t$ converges to zero pointwisely, as $t\rightarrow 0$. Moreover, $$\mu_{t+s}=\mu(a_{t+s})=\mu(a_s+a_{s+t}-a_s)=\mu(a_s)\boxplus\mu(a_{s+t}-a_{s})=\mu_s\boxplus\mu_t,$$ for $t, s>0$. Therefore, $\{\mu_t:t\ge 0\}$ is the semigroup of distributions corresponding to $\mu_{a_1}$ (see 4.5.3 in \cite{RS}). By Proposition 4.5.4 in \cite{RS}, $\kappa_n(a_1^{(i(1))}, \cdots, a_1^{(i(n))})=\lim_{t\rightarrow 0}\frac{1}{t}\tau(a_t^{(i(1))}, \cdots, a_t^{(i(n))})$. By the last part of the proof of 4.7.1 in \cite{RS}, we get $$\kappa_n(a_1^{(i(1))}, \cdots, a_1^{(i(n))})=\kappa_n(a_{i(1)}, \cdots, a_{i(n)}).$$

$(2)\Rightarrow (5)$. It is obvious from $(2)$ that $\kappa(PP^*)=\langle P, P\rangle\ge 0$, for $P\in \P$.
For $$X_{\underline{i}}=X_{i(1)}X_{i(2)} \cdots X_{i(m)}, X_{\underline{j}}=X_{j(1)}X_{i(2)} \cdots X_{j(n)}\in \P,$$ we divide every partition $\pi$ in $NC_{i,j}$ of all non-crossing partitions of the set $\{i(1), \cdots, i(m), j(1), \cdots, j(n)\}$ into three parts, $\pi=\pi_i\cup \pi_j\cup \pi_{i,j}$, where $\pi_i$ consists of subsets of $\{i(1), \cdots, i(m)\}$, $\pi_j$ consists of subsets of $\{j(1), \cdots, j(n)\}$, and each block $V$ of $\pi_{i,j}$ contains both $i$'s and $j$'s.  We then have
 \begin{align*}
&\kappa(X_{\underline{i}}X_{\underline{j}})\\
=&\kappa_{m+n}(a_{\underline{i}}, a_{\underline{j}})=\sum_{\pi\in NC_{i,j}}\varphi_{\pi_i}(a_{i(1)}, \cdots, a_{i(m)})\varphi_{\pi_j}(a_{j(1)}, \cdots, a_{j(n)})\\
\times &\prod_{V=\{i_1, \cdots, i_k, j_1, \cdots, j_l\}\in \pi_{i,j}}\varphi(a_{i_1}\cdots a_{i_k}a_{j_1}\cdots a_{j_l})\mu(\pi, 1_{m+n})\\
=&\sum_{\pi\in NC_{i,j}}\varphi_{\pi_j}(a_{j(1)}, \cdots, a_{j(n)})\varphi_{\pi_i}(a_{i(1)}, \cdots, a_{i(m)})\\
\times &\prod_{V=\{j_1, \cdots, j_l, i_1, \cdots, i_k\}\in \pi_{j,i}}\varphi(a_{j_1}\cdots a_{j_l}a_{i_1}\cdots a_{i_k})\mu(\pi, 1_{m+n})\\
=&\sum_{\pi\in NC_{j,i}}\varphi_{\pi}(a_{j(1)}, \cdots, a_{j(n)}, a_{i(1)}, \cdots, a_{i(m)})\mu(\pi, 1_{m+n})\\
=&\kappa_{m+n}(a_{j(1)},\cdots, a_{j(n)},  a_{i(1)}, \cdots, a_{i(m)})=\kappa(X_{\underline{j}}X_{\underline{i}}),
\end{align*}
where $NC_{j,i}$ is the set of all non-crossing partitions of $\{j(1), \cdots, j(n), i(1), \cdots, i(m)\}$, $$\pi=\pi_j\cup\pi_i\cup \pi_{j,i}$$ is a partition of $\pi\in NC_{j,i}$, similar to the  partition $\pi=\pi_i\cup \pi_j\cup \pi_{i,j}$ of a  partition $\pi$ in $NC_{i,j}$.
  It follows that
  $\kappa(PQ)=\kappa(QP)$, for all $P,Q\in \P$.
Moreover,
\begin{align*}
\kappa(X_{\underline{i}})&=\kappa_m(a_{i(1)}, a_{i(2)},\cdots, a_{i(m)})=\langle X_{i(1)}\cdots X_{i(m-1)}, X_{i(m)}\rangle\\
=&\overline{\langle X_{i(m)}, X_{i(1)}\cdots X_{i(m-1)}\rangle}=\overline{\kappa_m(a_{i(m)}, a_{i(m-1)}, \cdots, a_{i(1)})}\\
=&\overline{\kappa(X_{\underline{i}_r})}.
\end{align*}
Therefore,
   $$\kappa(P^*)=\sum\overline{\alpha_{i(1), \cdots, i(n)}}\kappa(X_{\underline{i}_r})=\overline{\kappa(P)}.$$

 $(5 )\Rightarrow (2)$ is obvious. Define $\langle P,Q\rangle=\kappa(PQ^*)$ for $P,Q\in \P$. Then, by $(5)$, $\langle \cdot,\cdot\rangle$ is non-negative sesquilinear on $\P$.
\end{proof}
By our multidimensional compound free Poisson limit theorem (Theorem 2.8), and above Theorem 3.3, we get the following corollary.
\begin{Corollary} If a sequence $\{b_n:n\in \mathbb{N}\}$ of self-adjoint operators in a $C^*$-probability space $(\A,\varphi)$
 has a multidimensional compound free Poisson distribution, then its distribution is multidimensional free infinitely divisible.
 \end{Corollary}
By (2) in Theorem 3.3, we have the following result.
 \begin{Corollary} Let $\{\{a_{m,n}:n=1, 2, \cdots\}: m=1, 2, \cdots\}$ be a sequence of sequences of random variables in a $C^*$-probability space $(\A, \varphi)$. If each sequence $\{a_{n,m}:n=1, 2, \cdots\}$ has a multidimensional free infinitely divisible distribution, for $m=1, 2, \cdots$, and $\{\{a_{m,n}: n=1, 2, \cdots\}: m=1, 2, \cdots\}$ converges in distribution to $\{b_n\}$ of random variables in a non-commutative probability space $(\B, \phi)$, as $m\rightarrow \infty$, then $\{b_n:n=1, 2, \cdots\}$ has a multidimensional free infinitely divisible distribution.
 \end{Corollary}

 In order to study multidimensional compound free Poisson distribution approximation to  multidimensional free infinitely divisible distributions, we need to make the following technical assumption.

 {\bf Assumption (A)} {\sl Let $\{b_1,b_2, \cdots \}$ be a sequence of random variables having a multidimensional free infinitely divisible distribution. By Theorem 4.5.5 in \cite{RS}, for each $n$, there is a sequence
 $\{p_{i, n}^{(j)}:i=1, 2, \cdots , n; j=1, 2, \cdots \}$ of random variables such that $\{p_{i, n}^{(j)}:i=1, 2, \cdots , n\}$ has a multidimensional compound free Poisson distribution, for all $j\ge 1$,  and the distribution of  $\{p_{i, n}^{(j)}:i=1, 2, \cdots , n\}$ converges to the distribution of $\{b_1, b_2, \cdots, b_n\}$, as $j\rightarrow \infty$. We say that the sequence $\{b_1, b_2, \cdots, \}$ satisfies Assumption (A) if random variables $p_{1,n}^{(j)}, p_{2, n}^{(j)}, \cdots, p_{n,n}^{(j)}$  can be chosen from a $C^*$-probability space $(\A_j, \varphi_j)$} with a faithful state $\varphi_j$, for $j=i,2,\cdots$.

 The following theorem generalizes Speicher's free compound Poisson distribution approximation theorem 4.5.5 in \cite{RS} to the case of infinitely dimensional distributions, under the above assumption.

\begin{Theorem} If a sequence $\{b_n:n=1, 2, \cdots\}$ of self-adjoint operators in a $C^*$-probability space $(\A,\varphi)$ has a multidimensional free infinitely divisible distribution and satisfies Assumption (A), then  there is  a bi-index sequence $\{p_{m,n}:n=1, 2, \cdots, m=1, 2, \cdots\}$ of self-adjoint operators in a $C^*$-probability space $(\B, \phi)$ such that $\{p_{m,n}: n=1, 2, \cdots\}$, for $m=1, 2, \cdots$, have multidimensional  compound free Poisson random distributions and $\{p_{m,n}:n=1, 2, \cdots\}$ converges in distribution to $\{b_n:n=1, 2, \cdots\}$, as $m\rightarrow\infty$.
\end{Theorem}
\begin{proof}
 Let $(\A, \varphi)$ be a $C^*$-probability space and  $\{b_n: n=1, 2, \cdots\}$ be a sequence of self-adjoint operators in $\A$ having a multidimensional  free infinitely divisible  distribution. Then  for each $m\in \mathbb{N}$, $\{b_1, \cdots, b_m\}$ has a multidimensional free infinitely divisible distribution. By 4.5.5 in \cite{RS}, there is a sequence $\{p_{i, m}^{(j)}: i=1, 2, \cdots, m, j=1, 2, \cdots\}$ of random variables such that $\{p_{1,m}^{(j)}, \cdots, p_{m,m}^{(j)}\}$ has a multidimensional compound free Poisson distribution, for $j=1, 2, \cdots$, and its distribution converges to that of $\{b_1, \cdots, b_m\}$, as $j\rightarrow \infty$. By Assumption (A), we can choose $p_{1, m}^{(j)}, \cdots,  p_{m,m}^{(j)}$ to be self-adjoint operators in a $C^*$-probability space $(\B_{m,j}, \varphi_{m,j})$ with faithful state $\varphi_{m,j}$.  Moreover, by the proof of Theorem 4.5.5 in \cite{RS}, the distribution of $\{p_{1, m}^j, \cdots, p_{m,m}^j\}$ is the same as that of $\{p_{1, m+1}^j, \cdots, p_{m,m+1}^j\}$, for $m\ge 1$ and $j=1, 2, \cdots$. Without loss of generality, we can assume $B_{m,j}$ is generated by $\{p_{i,m}^{(j)}: i=1, 2, \cdots, m\}$ and unit $1$. By Theorem 4.11 in \cite{NS}, there exists an isometric *-homomorphism $f_{m,j}: \B_{m,j}\rightarrow \B_{m+1, j}$ such that $f_{m,j}(p_{i,m}^{(j)})=p_{i, m+1}^{(j)}$, for $i=1, 2, \cdots, m$ for $m\ge 1$, for $j=1, 2, \cdots$. For $m<n$, define $f_{m,n,j}=f_{n-1, j}\circ\cdots \circ f_{m+1, j}\circ f_{m,j}$. Then $\{(\B_{m,j}, f_{m,j}): m=1, 2, \cdots\}$ is a directed system. We define an equivalent relation $\equiv$ in the disjoint union $\bigsqcup_{m=1}^\infty\B_{m,j}$ as follows.  For $a\in \B_{m,j}, b\in \B_{n,j}$, we say $a\equiv b$ if there is a $k>m, k>n$ such that $f_{m,k, j}(a)=f_{n,k, j}(b)$. For $a\in \bigsqcup_m\B_{m,j}$, we define $[a]$ to be  the equivalent class of $a$ in $\bigsqcup_m\B_{m,j}/\equiv$. Define a semi-norm $\|\cdot\|$ on $\bigsqcup_m\B_{m,j}/\equiv$ as follows.
 $$\|[a]\|=\lim_{n\rightarrow \infty}\|f_{m,n,j}(a)\|, $$ for $a\in \B_{m,j}$. We can get a $C^*$-algebra $\B_j$ from $(\bigsqcup_m\B_{m,j}/\equiv, \|\cdot\|)$ after dividing out the kernel $I=\{[x]\in \bigsqcup_m\B_{m,j}/\equiv: \|[x]\|=0\}$ and completion, which is called the directed limit $C^*$-algebra of $(\B_{m,j}, f_{m,j})$
 (see Pages 38-41 in \cite{KT}). Define a linear functional $\varphi_j: \B_j\rightarrow \mathbb{C}$, $\varphi_j([a])=\varphi_{m,j}(a)$, for $a\in \B_{m,j}$. Then $(\B_j, \varphi_j)$ is a $C^*$-probability space. Let $[p_i^{(j)}]=[p_{i, m}^{(j)}]$, for $m\ge i, i, j=1, 2, \cdots$ For $1\le i(1), i(2), \cdots, i(n)\le m$, we have
$$\varphi_j([p_{i(1)}^{(j)}]\cdots [p_{i(n)}^{(j)}])=\varphi_{m,j}(p_{i(1), m}^{(j)}\cdots p_{i(n),m}^{(j)}).$$ It follows that $\{[p_i^{(j)}]: i=1, 2, \cdots\}$ has a multidimensional compound free Poisson distribution, for $j=1, 2, \cdots$. Moreover,
$$\varphi(b_{i(1)}\cdots b_{i(n)})=\lim_{j\rightarrow \infty}\varphi_{m,j}(p_{i(1), m}^{(j)}\cdots p_{i(n),m}^{(j)})=\lim_{j\rightarrow \infty}\varphi_j([p_{i(1)}^{(j)}]\cdots [p_{i(n)}^{(j)}]).$$ The conclusion follows now.
\end{proof}

\end{document}